\documentclass[12pt,reqno]{amsart}

\newcommand\version{June 8, 2009}

\usepackage{amsmath,amsfonts,amsthm,amssymb,amsxtra}

\newtheorem{theorem}{Theorem}[section]
\newtheorem{proposition}[theorem]{Proposition}
\newtheorem{lemma}[theorem]{Lemma}

\theoremstyle{definition}

\newtheorem{assumption}[theorem]{Assumption}

\theoremstyle{remark}

\numberwithin{equation}{section}

\renewcommand{\epsilon}{\varepsilon}

\newcommand{\loc}{{\rm loc}}

\renewcommand{\phi}{\varphi}
\newcommand{\R}{\mathbb{R}}

\newcommand{\Sph}{\mathbb{S}}

\begin{document}

\title[Fractional Hardy inequalities --- \version]{Sharp fractional Hardy inequalities in half-spaces}

\author{Rupert L. Frank}
\address{Rupert L. Frank, Department of Mathematics,
Princeton University, Washington Road, Princeton, NJ 08544, USA}
\email{rlfrank@math.princeton.edu}

\author{Robert Seiringer}
\address{Robert Seiringer, Department of Physics, Princeton University,
P.~O.~Box 708,
       Princeton, NJ 08544, USA}
\email{rseiring@princeton.edu}

\thanks{Support through DFG grant FR 2664/1-1 (R.F.) and U.S. NSF grants PHY-0652854 (R.F.) and PHY-0652356 (R.S.)
is gratefully acknowledged.}

\dedicatory{Dedicated to V. G. Maz'ya}

\begin{abstract}
We determine the sharp constant in the Hardy inequality for fractional Sobolev spaces on half-spaces. Our proof relies on a non-linear and non-local version of the ground state representation.
\end{abstract}

\maketitle

\section{Introduction and main results}

This short note is motivated by the paper \cite{BD} concerning Hardy inequalities in the half-space $\R^N_+ := \{ (x',x_N):\ x'\in\R^{N-1}, x_N>0 \}$. The fractional Hardy inequality states that for $0<s<1$ and $1\leq p<\infty$ with $ps\neq 1$ there is a positive constant $\mathcal D_{N,p,s}$ such that
\begin{equation}\label{eq:mainintro}
 \iint_{\R^N_+\times\R^N_+} \frac{|u(x)-u(y)|^p}{|x-y|^{N+ps} } \,dx\,dy
\geq \mathcal D_{N,p,s} \int_{\R^N_+} \frac{|u(x)|^p}{x_N^{ps}}\,dx
\end{equation}
for all $u\in C_0^\infty(\overline{\R^N_+})$ if $ps<1$ and for all $u\in C_0^\infty(\R^N_+)$ if $ps>1$. In \cite{BD} the sharp (that is, the largest possible) value of the constant $\mathcal D_{N,2,s}$ for $p=2$ is calculated. Our goal in this paper is to determine the sharp constant $\mathcal D_{N,p,s}$ for \emph{arbitrary} $p$.

Indeed, we shall see that the sharp inequality \eqref{eq:mainintro} follows by a minor modification of the approach introduced in \cite{FrSe}. In that paper we calculated the sharp constant $\mathcal C_{N,p,s}$ in the inequality
\begin{equation}\label{eq:jfa}
 \iint_{\R^N\times\R^N} \frac{|u(x)-u(y)|^p}{|x-y|^{N+ps} } \,dx\,dy
\geq \mathcal C_{N,p,s} \int_{\R^N} \frac{|u(x)|^p}{|x|^{ps}}\,dx
\end{equation}
for all $u\in C_0^\infty(\R^N)$ if $1\leq p<N/s$ and for all $u\in C_0^\infty(\R^N\setminus\{0\})$ if $p>N/s$. A (non-sharp) version of \eqref{eq:jfa} was used by Maz'ya and Shaposhnikova \cite{MS} in order to simplify and extend considerably a result of Bourgain, Brezis and Mironescu \cite{BBM} on the norm of the embedding $\dot{W}^s_p(\R^N) \subset L_{Np/(N-ps)}(\R^N)$. Our proof of \eqref{eq:jfa} relied on a \emph{ground state substitution}, that is, on writing $u(x)=\omega(x) v(x)$ where $\omega(x)=|x|^{-(N-ps)/p}$ is a solution of the Euler-Lagrange equation corresponding to \eqref{eq:jfa}. In this note we shall prove \eqref{eq:mainintro} using that $\omega(x)=x_N^{-(1-ps)/p}$ satisfies the Euler-Lagrange equation corresponding to \eqref{eq:mainintro}.

We refer to \cite{BD,D,KMP} and the references therein for motivations and applications of fractional Hardy inequalities.

In order to state our main result let $1\leq p<\infty$ and $0<s<1$ with $ps\neq 1$ and denote by $\mathcal W^s_p(\R^N_+)$ the completion of $C_0^\infty(\R^N_+)$ with respect to the left side of \eqref{eq:mainintro}.
It is a consequence of the Hardy inequality that this completion is a space of functions. Moreover, it is well-known that for $ps<1$, $\mathcal W^s_p(\R^N_+)$ coincides with the completion of $C_0^\infty(\overline{\R^N_+})$.

\begin{theorem}[\textbf{Sharp fractional Hardy inequality}]\label{main}
 Let $N\geq 1$, $1\leq p<\infty$ and $0<s<1$ with $ps\neq 1$.  Then for all $u\in \mathcal W^s_p(\R^N_+)$,
\begin{equation}\label{eq:main}
 \iint_{\R^N_+\times\R^N_+} \frac{|u(x)-u(y)|^p}{|x-y|^{N+ps} } \,dx\,dy
\geq \mathcal D_{N,p,s} \int_{\R^N_+} \frac{|u(x)|^p}{x_N^{ps}}\,dx
\end{equation}
with
\begin{equation}\label{eq:mainconst}
\mathcal D_{N,p,s} := 2 \pi^{(N-1)/2} \frac{\Gamma((1+ps)/2)}{\Gamma((N+ps)/2)}
\int_0^1 \left|1 - r^{(ps-1)/p} \right|^p \frac{dr}{(1-r)^{1+ps}} \,.
\end{equation}
The constant $\mathcal D_{N,p,s}$ is optimal.
 If $p=1$ and $N=1$, equality holds iff $u$ is proportional to a non-increasing function. If $p>1$ or if $p=1$ and $N\geq 2$, the inequality is strict for any function $0\not\equiv u\in \mathcal W^s_p(\R^N_+)$.
\end{theorem}

For $p\geq 2$, inequality \eqref{eq:main} holds even with a remainder term.

\begin{theorem}[\textbf{Sharp Hardy inequality with remainder}]\label{remainder}
Let $N\geq 1$, $2\leq p<\infty$ and $0<s<1$ with $ps\neq 1$. Then for all $u\in \mathcal W^s_p(\R^N_+)$ and $v:= x_N^{(1-ps)/p} u$,
\begin{align}\label{eq:remainder}
\iint_{\R^N_+\times\R^N_+} \frac{|u(x)-u(y)|^p}{|x-y|^{N+ps} } \,dx\,dy
& - \mathcal D_{N,p,s} \int_{\R^N_+} \frac{|u(x)|^p}{x_N^{ps}}\,dx \notag \\
& \geq c_p\, \iint_{\R^N_+\times\R^N_+} \frac{|v(x)-v(y)|^p}{|x-y|^{N+ps} } \frac{dx}{x_N^{(1-ps)/2}} \frac{dy}{y_N^{(1-ps)/2}}
\end{align}
where $\mathcal D_{N,p,s}$ is given by \eqref{eq:mainconst} and $0<c_p\leq 1$ is given by
\begin{equation}\label{eq:gsrconst}
c_p:=\min_{0<\tau<1/2} \left((1-\tau)^p -\tau^p + p\tau^{p-1}\right)\,.
\end{equation}
If $p=2$, then \eqref{eq:remainder} is an equality with $c_2=1$.
\end{theorem}

We conclude this section by mentioning an open problem concerning fractional Hardy--Sobolev--Maz'ya inequalities. If $p\geq 2$ and $0<s<1$ with $1<ps<N$, is it true that the left side of \eqref{eq:remainder} is bounded from below by a positive constant times
$$
\left( \int_{\R^N_+} |u|^q \,dx \right)^{p/q},
\qquad q= Np/(N-ps) \,?
$$
The analogous estimate for $s=1$,
\begin{equation}\label{eq:hsm}
\int_{\R^N_+} |\nabla u|^p \,dx - \left(\frac{p-1}{p}\right)^p \int_{\R^N_+} \frac{|u(x)|^p}{x_N^{p}}\,dx
\geq \sigma_{N,p} \left( \int_{\R^N_+} |u|^q \,dx \right)^{p/q}\,,
\quad q= Np/(N-p)\,,
\end{equation}
is due to Maz'ya (for $p=2$) \cite{M} and Barbatis--Filippas--Tertikas (for $2<p<N$) \cite{BFT}; see also \cite{BFL} for the sharp value of $\sigma_{3,2}$. The proof of \eqref{eq:hsm} is based on the analogue of \eqref{eq:remainder},
$$
\int_{\R^N_+} |\nabla u|^p \,dx - \left(\frac{p-1}{p}\right)^p \int_{\R^N_+} \frac{|u(x)|^p}{x_N^{p}}\,dx
\geq c_p \int_{\R^N_+} |\nabla v|^p x_N^{p-1} \,dx \,,
\qquad u= x_N^{(p-1)/p} v \,.
$$


\section{Proofs}\label{sec:proofmain}

\subsection{General Hardy inequalities}

This subsection is a quick reminder of the results in \cite{FrSe}. Throughout we fix $N\geq 1$, $p\geq 1$ and an open set $\Omega\subset\R^N$. Let $k$ be a non-negative measurable function on $\Omega\times\Omega$ satisfying $k(x,y)=k(y,x)$ for all $x,y\in\Omega$ and define
$$
E[u] := \iint_{\Omega\times\Omega} |u(x)-u(y)|^p k(x,y) \,dx\,dy \ .
$$
Our key assumption for proving a Hardy inequality for the functional $E$ is the following.

\begin{assumption}\label{ass:el}
 Let $\omega$ be an a.e. positive, measurable function on $\Omega$. There exists a family of measurable functions $k_\epsilon$, $\epsilon>0$, on $\Omega\times\Omega$ satisfying $k_\epsilon(x,y)=k_\epsilon(y,x)$, $0\leq k_\epsilon(x,y) \leq k(x,y)$ and
\begin{equation}\label{eq:assk}
\lim_{\epsilon\to 0} k_\epsilon(x,y) = k(x,y)
\end{equation}
for a.e. $x,y\in\Omega$. Moreover, the integrals
\begin{equation}\label{eq:elreg}
 V_\epsilon(x) := 2\ \omega(x)^{-p+1} \int_{\Omega} \left(\omega(x)-\omega(y)\right) \left|\omega(x)-\omega(y)\right|^{p-2} k_\epsilon(x,y) \,dy
\end{equation}
are absolutely convergent for a.e. $x$, belong to $L_{1,\loc}(\Omega)$ and $V:=\lim_{\epsilon\to 0} V_\epsilon$ exists weakly in $L_{1,\loc}(\Omega)$, i.e., $\int V_\epsilon g\,dx \to \int V g\,dx$ for any bounded $g$ with compact support in $\Omega$.
\end{assumption}

The following abstract Hardy inequality was proved in \cite{FrSe} in the special case $\Omega=\R^N$. The general case considered here is proved by exactly the same arguments.

\begin{proposition}\label{hardy}
Under Assumption \ref{ass:el}, for any $u$ with compact support in $\Omega$ and $E[u]$ and $\int V_+ |u|^p \,dx$ finite one has
\begin{equation}\label{eq:hardy}
 E[u] \geq \int_{\Omega} V(x) |u(x)|^p \,dx \ .
\end{equation}
\end{proposition}

For $p\geq 2$, a stronger version of \eqref{eq:hardy} is valid which includes a remainder term.

\begin{proposition}\label{gsr}
Let $p\geq 2$. Under Assumption \ref{ass:el}, for any $u$ with compact support in $\Omega$ write $u=\omega v$ and assume that $E[u]$, $\int V_+ |u|^p \,dx$, and 
$$
E_\omega[v] := \iint_{\Omega\times\Omega} |v(x)-v(y)|^p \, \omega(x)^{\tfrac p2} k(x,y) \omega(x)^{\tfrac p2} \,dx\,dy
$$
are finite. Then
\begin{equation}\label{eq:gsr}
 E[u] - \int_{\Omega} V(x) |u(x)|^p \,dx \geq c_p \, E_\omega[v]
\end{equation}
with $c_p$ from \eqref{eq:gsrconst}. If $p=2$, then \eqref{eq:gsr} is an equality with $c_2=1$.
\end{proposition}


\subsection{Proof of Theorem \ref{main}}

Throughout this subsection we fix $N\geq 1$, $0<s<1$ and $p\neq 1/s$ and we abbreviate
$$
\alpha:=(1-ps)/p \, .
$$
We will deduce the sharp Hardy inequality \eqref{eq:main} using the general approach in the previous subsection with the choice
\begin{equation}\label{eq:abbreviate}
 \omega(x)= x_N^{-\alpha}\, ,
\quad
k(x,y) = |x-y|^{-N-ps}\, ,
\quad
V(x) = \mathcal D_{N,p,s} x_N^{-ps} \, .
\end{equation}
The key observation is

\begin{lemma}\label{eleq}
 One has uniformly for $x$ from compacts in $\R^N_+$
\begin{equation}\label{eq:eleq}
 2 \lim_{\epsilon\to 0} \int_{y\in\R^N_+\,, \left|x_N-y_N\right|>\epsilon } 
\left(\omega(x_N) -\omega(y_N)\right)
\left|\omega(x_N) - \omega(y_N) \right|^{p-2} k(x,y)\,dy
= \frac{\mathcal D_{N,p,s}}{x_N^{ps}} \ \omega(x)^{p-1}
\end{equation}
with $\mathcal D_{N,p,s}$ from \eqref{eq:mainconst}.
\end{lemma}

\begin{proof}
 First, let $N=1$. Then it follows from \cite[Lem. 3.1]{FrSe} that
$$
 2 \lim_{\epsilon\to 0} \int_{y>0\,, \left|x-y\right|>\epsilon }
\left(\omega(x) -\omega(y)\right)
\left|\omega(x) - \omega(y) \right|^{p-2} k(x,y)\,dy
= \frac{\mathcal D_{1,p,s}}{x^{ps}} \ \omega(x)^{p-1}
$$
uniformly for $x$ from compacts in $(0,\infty)$. To be more precise, in \cite[Lem. 3.1]{FrSe} the $y$-integral was extended over the whole axis. Therefore the difference between the constant $\mathcal C_{1,s,p}$ in \cite[(3.2)]{FrSe} and our $\mathcal D_{1,p,s}$ here comes from the absolutely convergent integral
$$
2 \int_{-\infty}^0
\left(\omega(x) -\omega(|y|)\right)
\left|\omega(x) - \omega(|y|) \right|^{p-2} \frac{dy}{(x-y)^{1+ps}} \,.
$$
This proves the assertion for $N=1$. In order to extend the assertion to higher dimensions we use the fact (see \cite[(6.2.1)]{AbSt}) that
\begin{align}\label{eq:bessel}
\int_{\R^{N-1}} \frac{dy'}{\left(|x'-y'|^2+m^2\right)^{(N+ps)/2}}
& = |\Sph^{N-2}| m^{-1-ps} \int_0^\infty \frac{r^{N-2} \,dr}{(r^2+1)^{(N+ps)/2}} \notag \\
& = \frac 12 |\Sph^{N-2}| m^{-1-ps} \frac{\Gamma((N-1)/2)\ \Gamma((1+ps)/2)}{\Gamma((N+ps)/2)}
\end{align}
for $N\geq 2$. Recalling $|\Sph^{N-2}|=2\pi^{(N-1)/2}/\Gamma((N-1)/2)$ concludes the proof.
\end{proof}

\begin{proof}[Proof of Theorem \ref{main}]
According to Lemma \ref{eleq}, Assumption \ref{ass:el} is satisfied with kernel $k_\epsilon(x,y)=|x-y|^{-N-ps}\chi_{\{|x_N-y_N|>\epsilon\}}$. Hence inequality \eqref{eq:main} for $u\in C_0^\infty(\R^N_+)$ follows from Proposition \ref{hardy}. By density it holds for all $u\in \mathcal W^s_p(\R^N_+)$. Strictness for $p>1$ follows by the same argument as in \cite{FrSe}. In order to discuss equality in \eqref{eq:main} for $p=1$ we first note that for equality it is necessary that $u$ is proportional to a non-negative function, which we assume henceforth. From \cite[(2.18)]{FrSe} we see that equality holds iff for a.e. $x$ and $y$ with $\omega(x_N)>\omega(y_N)$ (that is, $x_N<y_N$) one has
$$
|\omega(x_N) v(x) - \omega(y_N) v(y)| - \left( \omega(x_N) v(x) - \omega(y_N) v(y) \right) = 0
$$
for $v(x):=\omega(x_N)^{-1} u(x)$. Since for numbers $a,b\geq 0$ the equality $|a-b|-(a-b)=0$ holds iff $b\leq a$, we conclude that for a.e. $x$ and $y$ with $x_N<y_N$ one has $\omega(y_N) v(y) \leq \omega(x_N) v(x)$, that is $u(y)\leq u(x)$. If $N=1$ this means that $u$ is non-increasing. If $N\geq 2$ one sees that for a function $u$ with this property the integral $\int_{\R^N_+} |u| x_N^{-s} \,dx$ is infinite, unless $u\equiv 0$. This proves the strictness assertion in Theorem \ref{main}.

The fact that the constant is sharp for $N=1$ was shown in \cite{FrSe} (with $\R_+$ replaced by $\R$, but this only leads to trivial modifications). In order to prove sharpness in higher dimensions we consider functions of the form $u_n(x) = \chi_n(x') \phi(x_N)$, where
$$
\chi_n(x')=
\begin{cases}
1 & \text{if}\ |x'|\leq n\,,\\
n+1-|x'| & \text{if}\ n<|x'|<n+1\,,\\
0 & \text{if}\ |x'|\geq n+1\,.
\end{cases}
$$
An easy calculation using \eqref{eq:bessel} shows that
$$
\frac{\iint_{\R^N_+\times\R^N_+} \frac{|u_n(x)-u_n(y)|^p}{|x-y|^{N+ps} } \,dx\,dy}{\int_{\R^N_+} \frac{|u_n(x)|^p}{x_N^{ps}}\,dx } 
\to A \ 
\frac{\iint_{\R_+\times\R_+} \frac{|\phi(x_N)-\phi(y_N)|^p}{|x_N-y_N|^{1+ps} } \,dx_N\,dy_N}{\int_{\R_+} \frac{|\phi(x)|^p}{x_N^{ps}}\,dx_N } 
$$
as $n\to\infty$ with $A:=\frac 12 |\Sph^{N-2}| \Gamma((N-1)/2)\ \Gamma((1+ps)/2) / \Gamma((N+ps)/2)$. Since $A= \mathcal D_{N,p,s}/ \mathcal D_{1,p,s}$, sharpness of $\mathcal D_{N,p,s}$ for $N\geq 2$ follows from sharpness of $\mathcal D_{1,p,s}$ for $N=1$.
\end{proof}

\begin{proof}[Proof of Theorem \ref{remainder}]
Inequality \eqref{remainder} follows immediately from Proposition \ref{gsr}.
\end{proof}


\bibliographystyle{amsalpha}

\end{document}